\documentclass[sn-mathphys-num]{sn-jnl}

\usepackage{graphicx}%
\usepackage{multirow}%
\usepackage{amsmath,amssymb,amsfonts}%
\usepackage{amsthm}%
\usepackage{mathrsfs}%
\usepackage[title]{appendix}%
\usepackage{xcolor}%
\usepackage{textcomp}%
\usepackage{manyfoot}%
\usepackage{booktabs}%
\usepackage{algorithm}%
\usepackage{algorithmicx}%
\usepackage{algpseudocode}%
\usepackage{listings}%
\usepackage{tikz}
\usetikzlibrary{matrix,arrows,decorations.pathmorphing}
\usepackage{tikz-cd}

\usepackage{lmodern}
\usepackage[T1]{fontenc}
\usepackage{fancyhdr}
\usepackage{amssymb}
\usepackage{latexsym}
\usepackage{mathtools}
\usepackage{hyperref}
\usepackage{pdfpages}
\usepackage{verbatim}

\usepackage{amsthm}
\theoremstyle{plain}

\newtheorem{thm}{Theorem}

\newtheorem{prop}[thm]{Proposition}

\theoremstyle{definition}
\newtheorem{rem}[thm]{Remark}

\def\gcd{\mathrm{GCD}}
\def\lcm{\mathrm{lcm}}

\def\deg{\mathrm{deg}}
\def\aut{\mathrm{Aut}}
\def\div{\mathrm{div}}
\def\ord{\mathrm{ord}}

\def\cX{\mathcal{X}}
\def\cE{\mathcal{E}}
\def\K{\mathbb{K}}
\def\gg{g}

\theoremstyle{thmstyleone}%
\theoremstyle{thmstyletwo}%

\theoremstyle{thmstylethree}%

\begin{document}

\title[Algebraic curves with a large cyclic automorphism group]{Algebraic curves with a large cyclic automorphism group}


\author[1]{\fnm{Arianna} \sur{Dionigi}}\email{arianna.dionigi@unifi.it}

\author[2]{\fnm{Massimo} \sur{Giulietti}}\email{massimo.giulietti@unipg.it}

\author[2]{\fnm{Marco} \sur{Timpanella}}\email{marco.timpanella@unipg.it}

\affil[1]{\orgdiv{Department of Mathematics Ulisse Dini}, \orgname{University of Florence}, \orgaddress{\street{Via Giovanni Battista Morgagni, 67/a}, \city{Florence}, \postcode{50134}, \country{Italy}}}

\affil[2]{\orgdiv{Department of Mathematics and Computer Science}, \orgname{University of Perugia}, \orgaddress{\street{Via Luigi Vanvitelli, 1}, \city{Perugia}, \postcode{06123}, \country{Italy}}}



\abstract{The study of algebraic curves $\cX$ with numerous automorphisms in relation to their genus $g(\cX)$ is a well-established area in Algebraic Geometry. In 1995, Irokawa and Sasaki \cite{Sasaki} gave a complete classification of curves over $\mathbb{C}$ with an automorphism of order $N \geq 2g(\mathcal{X}) + 1$. 
Precisely, such curves are either hyperelliptic with $N=2g(\cX)+2$ with $g(\cX)$ even, or  are  quotients of the Fermat curve of degree $N$ by a cyclic group of order $N$.
Such a  classification does not hold in positive characteristic $p$, the curve with equation $y^2=x^p-x$ being a well-studied counterexample. 
This paper successfully classifies curves with a cyclic automorphism group of order $N$ at least $2g(\mathcal{X}) + 1$ in positive characteristic $p \neq 2$, offering the positive characteristic counterpart to the Irokawa-Sasaki result.
The possibility of wild ramification in positive characteristic has presented a few challenges to the investigation.}

\keywords{Algebraic curves, positive characteristic, automorphism groups.}


\pacs[MSC Classification]{14H37, 14H05}

\maketitle

\section{Introduction}\label{sec1}

In this paper, let $\mathcal{X}$ denote a projective, geometrically irreducible, nonsingular algebraic curve defined over an algebraically closed field $\mathbb{K}$ with characteristic $p$. We denote the automorphism group of $\mathcal{X}$ that fixes $\mathbb{K}$ elementwise by $\aut(\mathcal{X})$.

By a classical result, $\aut(\mathcal{X})$ is finite if the genus $g(\mathcal{X})$ of $\mathcal{X}$ is at least two; see, for example, \cite[Chapter 11]{HKT}. It is also well known that when the characteristic $p$ divides $|\aut(\mathcal{X})|$, the classical Hurwitz bound
$$|\aut(\mathcal{X})| \leq 84(g(\mathcal{X})-1)$$
does not generally apply. Specifically, if $g(\mathcal{X}) \geq 2$, then $|\aut(\mathcal{X})| < 8g(\mathcal{X})^3$ except for four infinite families of curves; see \cite{henn} or \cite[Theorem 11.127]{HKT}. Although algebraic curves defined over fields of positive characteristic can exhibit many more automorphisms than predicted by the Hurwitz bound, stronger bounds for subgroups $G$ of $\aut(\mathcal{X})$ under certain conditions have been established. For instance,
\begin{itemize}
\item for solvable groups $G$ and odd $p$, $|G| \leq 34(g(\mathcal{X})+1)^{3/2}$; see \cite{kmordinarie};
\item for solvable groups $G$, even $g(\mathcal{X})$, and $p = 2$, $|G| \leq 35(g(\mathcal{X})+1)^{3/2}$; see \cite{msordinarie};
\item for abelian groups, $|G| \leq 4g(\mathcal{X}) + 4$; see Theorem \ref{theorem11.79};
\item for subgroups of prime power order $d$ with $d = p$, $|G| \leq 4(g(\mathcal{X})-1)$; see \cite{NAKA};
\item for subgroups of prime power order $d$, $p$ odd, and $d \neq p$, $|G| \leq 9(g(\mathcal{X})-1)$; see \cite{kmoddpower};
\item for groups $G$ fixing a point, $|G| \leq \frac{4pg(\mathcal{X})^2}{p-1}\left( \frac{2g(\mathcal{X})}{p-1}+1\right)$; see \cite[Theorem 2.1]{Singh}. See also \cite{AdvancesGK, LiaTimpanella} for other bounds on automorphism groups fixing a point.
\end{itemize}
In the complex case, if $G$ is a cyclic group, then $|G| \leq 4g(\mathcal{X}) + 2$. A complete characterization of algebraic curves over $\mathbb{C}$ with an automorphism of order at least $2g(\mathcal{X}) + 1$ was given by Irokawa and Sasaki \cite{Sasaki}. They proved that if $\mathcal{X}$ is a curve of genus at least 2 and its automorphism group contains a cyclic subgroup $G$ of order $N \geq 2g(\mathcal{X}) + 1$, then $\mathcal{X}$ is either isomorphic to the curve
\begin{equation*}
y^N = x^r(1-x)^s,
\end{equation*}
for some $r, s \geq 1$, with $r + s \leq N - 1$ and $\gcd(r, s, N) = 1$, or to the curve
\begin{equation*}
y^2 = (x^{g(\mathcal{X})+1} - 1)(x^{g(\mathcal{X})+1} - \lambda),
\end{equation*}
where $\lambda \in \mathbb{K} \setminus \{0, 1\}$. In the first case, $\mathcal{X}$ is a quotient of the Fermat curve of degree $N$, whereas in the second case, it is a hyperelliptic curve.

As a matter of terminology, we say that $\mathcal X$ is an $N$-curve if it admits an automorphism group of size $N$.
In arbitrary characteristic, so far only the prime order case has been thoroughly investigated:
Homma \cite{Homma} proved that if $N=q$ is a prime, then either $q\le g(\cX)+1$ or $q=2g(\cX)+1$ and provided a classification for the latter case; more recently, Arakelian and Speziali \cite{Arakelian} classified $(g(\cX)+1)$-curves, and fully characterized the automorphism groups of $q$-curves
for $q=2g(\cX)+1$ and for $q=g(\cX)+1$.

In this paper, we tackle the problem of classifying $N$-curves with $N \geq 2g(\cX) + 1$ in positive characteristic. The challenge in this context, compared to the complex case, arises from the potential for wild ramification, particularly due to the presence of $p$-elements in $G$ that fix a point on $\mathcal{X}$.
We achieve success for $p \neq 2$, and our main result is as follows.
\begin{thm}\label{MAIN}
Let $p \geq 3$, and let $\mathcal{X}$ be a curve of genus $g(\mathcal{X}) \geq 2$ with a cyclic automorphism group $G$ of order $N \geq 2g(\mathcal{X}) + 1$. Then, up to birational equivalence, one of the following holds:
\begin{itemize}
\item[(I)] $p \nmid N$, and either 
\begin{equation*}
\mathcal{X}: y^N = x^r(1-x)^s,
\end{equation*}
for some natural numbers $r$ and $s$ such that $r + s \leq N - 1$ and $\gcd(r, s, N) = 1$, or $\mathcal{X}$ is a hyperelliptic curve
\begin{equation*}
\mathcal{X}: y^2 = \big(x^{\frac{N}{2}} - 1\big)\big(x^{\frac{N}{2}} - \lambda\big),
\end{equation*}
where $\lambda \in \mathbb{K} \setminus \{0, 1\}$ and $N\equiv 2 \pmod 4$;
\item[(II)] $N = pm$ with $m > 1$, $(p,m)=1$, and $p \neq 3$; also, either $g(\mathcal{X}) = \frac{(p-1)(m-1)}{2}$ and
\begin{equation*}
\mathcal{X}: y^p - y = a(x^m - b)
\end{equation*}
for some $a, b \in \mathbb{K}$, or $g(\mathcal{X}) = p - 1$ and
\begin{equation*}
\mathcal{X}: by^p + cy = ax + \frac{1}{x}
\end{equation*}
for some $a, b, c \in \mathbb{K}$.
\item[(III)] $N = p$, and
\begin{equation*}
\mathcal{X}: y^p - y = x^2.
\end{equation*}
\end{itemize}
\end{thm}

The paper is structured as follows. 
First, Section \ref{background} provides the necessary background on algebraic curves and proofs of some preliminary results. 
Then we present the proof of our main result. Due to the possible occurrence of wild ramification, we will divide the proof into two sections: Section \ref{noncoprimo} and Section \ref{coprimo}, according to whether $|G|$ is divisible by $p$ or not. 

\section{Background and Preliminary Results}\label{background}

In this section, we review the known results on automorphism groups that are pertinent to the rest of the paper. Our notation and terminology are standard; see \cite[Chapter 11]{HKT} for details.
Let $\mathcal{X}$ denote a projective, geometrically irreducible, nonsingular algebraic curve defined over an algebraically closed field $\mathbb{K}$ with characteristic $p$, and let $\aut(\mathcal{X})$ be its automorphism group.
Let $\mathbb K(\mathcal X)$ be the function field of $\mathcal X$. The $p$-rank of $\mathcal{X}$, also known as the Hasse-Witt invariant of $\mathcal{X}$, is defined as the rank of the (elementary abelian) group of $p$-torsion points in the Jacobian variety of $\mathcal{X}$. If $\gamma(\mathcal{X})$ denotes the $p$-rank of $\mathcal{X}$, then we have $\gamma(\mathcal{X}) \leq {\gg}(\mathcal{X})$, and equality holds if and only if $\mathcal{X}$ is an ordinary curve.

For a subgroup $G$ of $\aut(\cX)$, let $\bar \cX$ be a non-singular model of $\K(\cX)^G$, that is,
a projective non-singular geometrically irreducible algebraic
curve with function field $\K(\cX)^G$, where $\K(\cX)^G$ consists of all elements of $\K(\cX)$
fixed by every element in $G$. Usually, $\bar \cX$ is called the
quotient curve of $\cX$ by $G$ and denoted by $\cX/G$. The field extension $\K(\cX)|\K(\cX)^G$ is Galois of degree $|G|$. Let $\Phi$ denote the associated rational map $\cX\to \bar{\cX}$.

A point $P \in \mathcal{X}$ is called a ramification point of $G$ if the stabilizer $G_P$ of $P$ in $G$ is non-trivial; the ramification index $e_P$ at $P$ is given by $|G_P|$. A point $\bar{Q}\in\bar{\cX}$ is a branch point of $G$ if there is a ramification point $P\in \cX$ such that $\Phi(P)=\bar{Q}$;

The $G$-orbit of $P \in \mathcal{X}$ is the subset of $\mathcal{X}$
$$o = \{ R \mid R = g(P), \, g \in G \},$$
and it is called {\em long} if $|o| = |G|$; otherwise, $o$ is {\em short}. For a point $\bar{Q}$, the $G$-orbit $o$ lying over $\bar{Q}$ consists of all points $P \in \mathcal{X}$ such that $\Phi(P) = \bar{Q}$. If $P\in o$ then $|o|=|G|/|G_P|$ and hence $\bar{Q}$ is a branch point if and only if $o$ is a short $G$-orbit. Note that it is possible for $G$ to have no short orbits, which occurs if and only if every non-trivial element of $G$ is fixed-point-free on $\mathcal{X}$, meaning the cover $\Phi$ is unramified. On the other hand, $\mathcal{X}$ has a finite number of short $G$-orbits.

For a non-negative integer $i$, the $i$-th ramification group of $\mathcal{X}$ at $P$ is denoted by $G_P^{(i)}$ and is defined as
$$G_P^{(i)} = \{ g \mid \ord_P(g(t) - t) \geq i + 1, g \in G_P \},$$
where $t$ is a uniformizing element (local parameter) at $P$. Here, $G_P^{(0)} = G_P$. The structure of $G_P$ is well understood; see, for example, \cite[Chapter IV, Corollary 4]{Serre} or \cite[Theorem 11.49, Theorem 11.60]{HKT}.
\begin{thm}\label{resultstabilizz} 
The stabilizer $G_P$ of a point $P \in \mathcal{X}$ in $G$ has the following properties:
\begin{itemize}
\item[\rm(i)] $G_P^{(1)}$ is the unique Sylow $p$-subgroup of $G_P$;
\item[\rm(ii)] For $i \geq 1$, $G_P^{(i)}$ is a normal subgroup of $G_P$ and the quotient group $G_P^{(i)}/G_P^{(i+1)}$ is an elementary abelian $p$-group;
\item[\rm(iii)] $G_P = G_P^{(1)} \rtimes U$, where $U$ is a cyclic group whose order is prime to $p$.
\end{itemize}
\end{thm}

The Hurwitz genus formula applied to $G$ is given by:
\begin{equation}
\label{Hurwitz}
2g(\mathcal{X}) - 2 = |G|(2g(\bar{\mathcal{X}}) - 2) + \sum_{P \in \mathcal{X}} d_P,
\end{equation}
where
\begin{equation}
\label{differente}
d_P = \sum_{i \geq 0} (|G_P^{(i)}| - 1).
\end{equation}

A subgroup of $\aut(\mathcal{X})$ is a $p'$-group (or a prime-to-$p$ group) if its order is prime to $p$. A subgroup $G$ of $\aut(\mathcal{X})$ is {\em tame} if the stabilizer of any point in $G$ is a $p'$-group. Otherwise, $G$ is {\em non-tame} (or {\em wild}). While every $p'$-subgroup of $\aut(\mathcal{X})$ is tame, the converse is not always true.

\begin{thm}\cite[Theorem 11.60]{HKT} \label{theorem11.60HKT}
If $G$ is a subgroup of $\aut(\mathcal{X})$ that fixes a point and $p \nmid |G|$, then
$$|G| \leq 4g(\mathcal{X}) + 2.$$
\end{thm}

\begin{thm}\cite[Theorem 11.79]{HKT} \label{theorem11.79}
If $G$ is an abelian subgroup of $\aut(\mathcal{X})$, then
$$|G| \leq \begin{cases}
4g(\mathcal{X}) + 4, & \text{if } p \neq 2, \\
4g(\mathcal{X}) + 2, & \text{if } p = 2.
\end{cases}$$
\end{thm}
If a curve has genus $0$ it is called rational and its automorphism group is isomorpihc to $\mathrm{PGL}(2, \mathbb{K})$.
The following result is a consequence of Dickson's classification of finite subgroups of $\mathrm{PGL}(2, \mathbb{K})$ (see \cite{Dickson}; also \cite[Theorem 11.91]{HKT}).
\begin{prop}
\label{dickson}
Any non-trivial tame automorphism of a rational curve that fixes a point has exactly two fixed points.
\end{prop}
We will also use additional information on automorphisms of elliptic curves, i.e. curves of genus $1$.
Let $\cE$ be a non-singular plane cubic curve viewed as a birational model of an elliptic curve.
For an inflection point $O$ of $\cE$, the set of points of $\cE$ can be equipped by an operation $\oplus$ to form an abelian group $G_O$ with zero-element $O$, which is isomorphic to the zero Picard group of $\cE$; see for instance \cite[Theorem 6.107]{HKT}. The translation $\tau_P$ associated with $P\in \cE$ is the permutation on the points of $\cE$ with equation $\tau_P: X\mapsto X\oplus P$. Since there exists an automorphism in $\aut(\cE)$ which acts on $\cE$ as $\tau_P$ does, translations of $\cE$ can be viewed as elements of $\aut(\cE)$. They form the translation group $T$ of $\cE$ which acts faithfully on $\cE$ as a sharply transitive permutation group. For every prime $r$, the elements of order $r$ in $T$ are called $r$-torsion points. They together with the identity form an elementary abelian $r$-group of rank $k$. Here $k=2$ for $r\neq p$, whereas $k$ equals the $p$-rank of $\cE$ for $r=p$, that is, $k=0,1$ according as $\cE$ is supersingular or ordinary.
\begin{prop}
\label{strutturaell} If $\cE$ is an elliptic curve, the translation group $T$ is a normal subgroup of $\aut(\cE)$, and $\aut(\cE)=T\rtimes \aut(\cE)_P$ for every $P\in \cE$.
\end{prop}
Concerning automorphisms of elliptic curves fixing a point, the following result holds.
\begin{thm}[{\cite[Theorem 10.1]{silverman2009}, \cite[Theorem 11.94]{HKT}}]\label{autelliptic}
Let $\mathcal{E}$ be an elliptic curve. If $G$ is a subgroup of $\aut(\mathcal{E})$ and $P \in \mathcal{E}$, then
\begin{equation}
|G_P| = \left\{
\begin{array}{ll}
2, 4, 6 & \text{when } p \neq 2, 3, \\
2, 4, 6, 12 & \text{when } p = 3, \\
2, 4, 6, 8, 12, 24 & \text{when } p = 2.
\end{array}
\right.
\end{equation}
\end{thm}

\begin{prop}\label{ellittica}
Let $\mathcal{E}$ be an elliptic curve defined over a field of characteristic $p > 2$, and let $\alpha$ be an automorphism of $\mathcal{E}$ of finite order $o(\alpha)$. 
\begin{itemize}
\item[(a)] If $o(\alpha) = p^h$, then either $\mathcal{E}$ is not ordinary, $p^h = 3$, and $\alpha$ does not act semiregularly on $\mathcal{E}$, or $\mathcal{E}$ is ordinary and $\alpha$ acts semiregularly on $\mathcal{E}$.

\item[(b)] If $(o(\alpha), p) = 1$ and $\alpha$ does not act semiregularly on $\mathcal{E}$, then $\alpha$ has order in $\{2, 3, 4, 6\}$. Specifically:
\begin{itemize}
\item If $o(\alpha) = 2$, then $\alpha$ fixes 4 points and acts semiregularly on the other points of $\mathcal{E}$.
\item If $o(\alpha) = 3$, then $\alpha$ fixes 3 points and acts semiregularly on the other points of $\mathcal{E}$.
\item If $o(\alpha) = 4$, then $\alpha$ fixes 2 points, and the only other short orbit has size 2.
\item If $o(\alpha) = 6$, then $\alpha$ fixes 1 point, and the only other short orbits are two orbits of sizes 2 and 3.
\end{itemize}
\end{itemize}
\end{prop}

\begin{proof}
\begin{itemize}
    \item[(a)] If $\mathcal{E}$ is not ordinary, then $\langle\alpha\rangle$ fixes a point and acts semiregularly on the other points. By \cite[Theorem 11.94]{HKT}, we have $h = 1$ and $p = 3$.

    Assume now that $\mathcal{E}$ is ordinary. Then, by Proposition \ref{autelliptic}, $\langle\alpha\rangle T/T \cong \langle\alpha\rangle / (\langle\alpha\rangle \cap T)$ is either trivial or a $p$-group and is isomorphic to a subgroup of the stabilizer of $O$. Note that $p$ does not divide the size of the stabilizer of $O$; this follows directly from \cite[Theorem 11.94]{HKT} for $p > 3$. For $p = 3$, it can be checked that 3 divides the size of the stabilizer of $O$ only if the curve is not ordinary. Therefore, $\langle\alpha\rangle = \langle\alpha\rangle \cap T$ and $\alpha$ belongs to $T$, acting semiregularly on $\mathcal{E}$.

    \item[(b)] Let $m = o(\alpha)$. Since $\alpha$ does not act semiregularly on $\mathcal{E}$, the genus of the quotient curve $\mathcal{E} / \langle\alpha\rangle$ is 0. Applying the Riemann-Hurwitz genus formula for the group generated by $\alpha$, with $\Omega_1, \ldots, \Omega_n$ denoting the short orbits under the action of $\alpha$, we have:
    \begin{equation*}
    0 = -2m + \sum_{i=1}^n (m - |\Omega_i|) \geq -2m + \frac{m}{2}n.
    \end{equation*}
    Thus, $3 \le n \le 4$. If $n = 3$, then:
    \begin{equation*}
    m = |\Omega_1| + |\Omega_2| + |\Omega_3|.
    \end{equation*}
    Considering the equation $\frac{1}{e_1} + \frac{1}{e_2} + \frac{1}{e_3} = 1$, where $e_i$ is the size of the stabilizer of a point in $\Omega_i$ and assuming $e_1 \le e_2 \le e_3$, and using \cite[Theorem 11.94]{HKT}, the possible values for $e_i$ are $\{2, 3, 4, 6\}$. If $e_i = 3$ or $6$ for some $i$, then $p > 3$. The only possibilities are:
    \begin{itemize}
    \item $e_1 = 2$, $e_2 = e_3 = 4$; hence, $4 \mid m$, and an element of order 2 in $\langle\alpha\rangle$ fixes $m$ points. By \cite[Lemma 11.107]{HKT}, we get $m = 4$.
    \item $e_1 = e_2 = e_3 = 3$; hence, $3 \mid m$, and an element of order 3 in $\langle\alpha\rangle$ fixes $m$ points. By \cite[Lemma 11.106]{HKT}, this implies $m = 3$.
    \item $e_1 = 2$, $e_2 = 3$, $e_3 = 6$; hence, $6 \mid m$, and the elements of orders 2, 3, and 6 in $\alpha$ fix $\frac{2m}{6}$, $\frac{m}{3} + \frac{m}{6}$, and $\frac{m}{6}$ points respectively. By \cite[Lemma 11.106]{HKT}, this implies $m = 6$.
    \end{itemize}
     If $n = 4$, then:
    \begin{equation*}
    2m = |\Omega_1| + |\Omega_2| + |\Omega_3| + |\Omega_4|.
    \end{equation*}
    The only possibility is $|\Omega_i| = m / 2$ for each $i$, and hence the element of order 2 in $\alpha$ fixes $2m$ points. By \cite[Lemma 11.107]{HKT}, this implies $m = 2$.
\end{itemize}
\end{proof}

We need the following corollary to Theorem 1.5 from \cite{MZ}. Although stated for the algebraic closure of a finite field, this result holds for every algebraically closed field of positive characteristic $p$.

\begin{prop}\label{MZ}
Let $\mathcal{X}$ be a curve of genus $g(\mathcal{X}) = q - 1$, where $q$ is a power of $p$. If $\operatorname{Aut}(\mathcal{X})$ contains an elementary abelian subgroup $E_q$ of order $q$, then, up to birational equivalence over $\mathbb{K}$, one of the following holds:
\begin{itemize}
    \item[(a)] $\mathcal{X}$ has an affine equation
    \begin{equation*}
    L_1(y) = ax + \frac{1}{x},
    \end{equation*}
    where $a \in \mathbb{K}^*$ and $L_1(T) \in \mathbb{K}[T]$ is a separable $p$-linearized polynomial of degree $q$. Furthermore, $\mathcal{X}$ is ordinary and $|\operatorname{Aut}(\mathcal{X})| = 4q$.
    \item[(b)] If $p \neq 3$, then $\mathcal{X}$ has an affine equation
    \begin{equation*}
    L_2(y) = x^3 + bx,
    \end{equation*}
    where $b \in \mathbb{K}$ and $L_2(T) \in \mathbb{K}[T]$ is a separable $p$-linearized polynomial of degree $q$. Furthermore, the $p$-rank of $\mathcal{X}$ is equal to zero.
\end{itemize}
\end{prop}
We now prove a basic result  on Kummer extensions that will be used throughout the paper.
\begin{prop}\label{n=1}
The number of branch points of a Kummer extension is different from $1$.
\end{prop}
\begin{proof}
Let $\mathbb{K}(\mathcal{X})/\mathbb{K}\left(\mathcal Y\right)$
be a Kummer extension of degree $N$; that is $\mathbb{K}(\mathcal{X})=\mathbb{K}\left(\mathcal Y\right)(y)$ with
$$
y^N - \alpha = 0 \quad \text{and} \quad \alpha \in \mathbb{K}\left(\mathcal{Y}\right).
$$
From Kummer theory (see e.g. \cite[Proposition 3.7.3 (b)]{Stichtenoth}), the branch points of this extension are the points $P$ of $\mathcal Y$ for which $\gcd(\ord_P(\alpha), N) < N$, i.e., $N \nmid \ord_P(\alpha)$. Suppose by contradiction that there is a unique branch point $P_0$. Then,
$$
0 = \deg(\div(\alpha)) = \ord_{P_0}(\alpha) + \sum_{P \neq P_0} \ord_P(\alpha).
$$
Thus,
\begin{equation}\label{equazione}
-\ord_{P_0}(\alpha) = \sum_{P \neq P_0} \ord_P(\alpha),
\end{equation}
which is a contradiction because $N$ divides the right-hand side of \eqref{equazione} but not the left-hand side.
\end{proof}

From now on, $\cX$ is a non-singular algebraic curve of genus $g(\cX)\geq 2$ defined over an algebraically closed field $\mathbb{K}$ of characteristic $p\neq 2$. Also, we will denote by $\sigma$ an automorphism of $\mathcal{X}$ of order $N$, and by $G=\langle\sigma\rangle$ the subgroup of $\aut(\cX)$ generated by $\sigma$.
We conclude this section with  two preliminary results about ramification points in the extension $\mathcal{X}$ over $\mathcal{X}/G$.
\begin{prop}\label{n=0} If $N\geq 2g(\cX)+1$, the extension $\cX$ over $\mathcal{X}/G$ ramifies.
\end{prop}
\begin{proof}
Suppose by contradiction that the extension is unramified. Then, by the Riemann-Hurwitz formula applied to $G$ we have
$$2g(\cX)-2=N(2g_0-2),$$
where $g_0$ is the genus of $\mathcal{X}/G$.
If $g_0>1$ then $2g(\cX)-2\geq 2N$ that is $N\leq g(\cX)-1$ against $N\geq 2g(\cX)+1$; if $g_0=1$ then $2g(\cX)-2=0$ that is $g(\cX)=1$ against $g(\cX)\geq 2$; if $g_0=0$ then $2g(\cX)-2=-2N$, again a contradiction.
\end{proof}

\begin{prop}\label{divisibilità} Let $d$ be a divisor of $N$ and $H$ be the subgroup of $G$ of order $d$. Denote by $\mathcal{Y}$ and $\mathcal{Z}$ the quotient curves $\mathcal{X}/H$ and $\mathcal{X}/G$, respectively. For a point $P\in\mathcal{X}$, let $Q\in \mathcal{Y}$ and $R\in \mathcal{Z}$ be the points lying under $P$ with respect to the coverings $\mathcal{X}\to\mathcal{Y}$ and $\mathcal{X}\to\mathcal{Z}$. Then, $R$ is a branch point for the covering $\mathcal{Y}\to\mathcal{Z}$ if and only if $|G_P|\nmid d$.
\end{prop}
\begin{proof}
Assume that $e_P=|G_P| \mid d$. Then, $H$ contains a unique subgroup of order $e_P$, and this subgroup is $G_P$. Therefore, $G_P=H_P$ must hold, and hence the covering $\mathcal{Y}\to\mathcal{Z}$ is not ramified over $R$. 
On the other hand, if $e_P\nmid d$, $H$ has no subgroups of order $e_P$, thus $G_P\not\subseteq H$. Therefore, $H_P$ is properly contained in $G_P$, and so $\mathcal{Y}\to\mathcal{Z}$ ramifies over $R$.
\end{proof}

\section{Proof of Theorem \ref{MAIN} when $p$ divides $N$}\label{noncoprimo}

We keep the notation of the last part Section \ref{background}. That is,
$\cX$ is a non-singular algebraic curve of genus $g(\cX)\geq 2$ defined over an algebraically closed field $\mathbb{K}$ of characteristic $p\neq 2$. Also,  $\sigma$ is an automorphism of $\mathcal{X}$ of order $N$ and $G=\langle\sigma\rangle$  is the subgroup of $\aut(\cX)$ generated by $\sigma$. 

We assume that $N\ge 2g(\mathcal X)+1$ and that $p$ divides $N$. Our goal in this Section is to prove that either (II) or (III) of Theorem \ref{MAIN} holds. To this end, we consider the cases $N=p^h$ and $N=p^hm$ with $(p,m)=1$, $m>1$, in two different subsections.

\subsection{The case $N=p^h$}

First, we show that that $h\le 2$ holds under the weaker condition $N\ge 2g(\cX)-1$.
\begin{prop}\label{10mar}
    If $N=p^h$ with $p\ge 3$ and $N\ge 2g(\cX)-1$, then  $h\le 2$. 
\end{prop}
    \begin{proof}
    If $p>3$ then the $p$-rank $\gamma(\cX)$ of $\mathcal X$ is equal to zero by \cite[Theorem 11.84]{HKT}. Actually, we may assume $\gamma(\cX)=0$ also for $p=3$. In fact, by  \cite[Proposition 3.3]{Large3}, if $\gamma(\cX)>0$ and $3^h>2g(\cX)-2$ then  $h=1$.         
    
    Therefore, by \cite[Lemma 11.129]{HKT} and \cite[Remark 11.128]{HKT} $G$ fixes one point $P_0\in \mathcal X$ and acts semiregularly  on the other points of the curve. By the Riemann-Hurwitz formula applied to $G$,
$$
2g(\cX)-2=p^h(2g_0-2)+\sum_{i=0}^{\infty} (|G_{P_0}^{(i)}|-1),$$
where $g_0$ is the genus of $\cX/G$.
Since $G=G_{P_0}^{(0)}=G_{P_0}^{(1)}$ is cyclic, by (iii) of \cite[Theorem 11.74]{HKT} we have that the factor group $G_{P_0}^{(i)}/G_{P_0}^{(i+1)}$
is either trivial or a cyclic group of order $p$. Also, by (v) of \cite[Lemma 11.75]{HKT} the integers $k$ for which $G_{P_0}^{(k)}/G_{P_0}^{(k+1)}$ is non-trivial are all congruent modulo $p$.

Assume that $h\ge 3$. Then
$$
\sum_{i=0}^{\infty} (|G_{P_0}^{(i)}|-1)
$$
is at least
$$
2(p^h-1)+p(p^{h-1}-1)+p(p^{h-2}-1)+\ldots+p(p-1)
$$
that is
$$
2p^h-2+p^h+p^{h-1}+\ldots+p^2-(h-1)p.
$$
As $p^h\ge 2g(\cX)-1$
$$
2g(\cX)-2\ge p^h+p^{h-1}+\ldots+p^2-(h-1)p-2\ge 2g(\cX)-1+p^{h-1}+\ldots+p^2-(h-1)p-2
$$
which implies
$$
p^{h-1}+\ldots+p^2-(h-1)p-1\le 0,
$$
a contradiction.

    \end{proof}
Now,  the case $N=p^2$ is considered. We are going to show that it is not possible that $p^2\ge 2g(\cX)+1$. Our first step is to rule out all possibilities but  $p^2=2g(\cX)+p$.
\begin{prop}\label{13mar}
    If $N=p^2$ with $p\ge 3$ and $N\ge 2g(\cX)+1$, then $p^2=2g(\cX)+p$. Also, there is a unique short orbit $\Omega=\{P_0\}$ of $\mathcal X$ under the action of $G$, and
$$
|G_{P_0}^{(0)}|=|G_{P_0}^{(1)}|=p^2,\qquad |G_{P_0}^{(2)}|=\ldots=|G_{P_0}^{(p+1)}|=p,\qquad |G_{P_0}^{(p+2)}|=1.
$$
Moreover, the quotient curve $\mathcal X/G_{P_0}^{(2)}$ is rational.
 \end{prop}
\begin{proof} By the proof of Proposition \ref{10mar},
$G$ fixes one point $P_0\in \mathcal X$ and acts semiregularly  on the other points of the curve. Also, there exist two positive integers $v\geq 2$ and $w$ such that $|G_{P_0}^{(v-1)}|=p^2$ and
$$
p=|G_{P_0}^{(v)}|=|G_{P_0}^{(v+1)}|=\ldots=|G_{P_0}^{(wp+v-1)}|\neq |G_{P_0}^{(wp+v)}|=1.
$$

Let $T=G_{P_0}^{(v)}$. Clearly $T_{P_0}^{(i)}=T$ for each $i\le wp+v-1$. Write the Riemann-Hurwitz genus formula for $T$: 
$$
2g(\cX)-2=p(2g(\mathcal X/T)-2)+(wp+v)(p-1).
$$
Then
$$
p^2-3\ge 2g(\cX)-2= p(2g(\mathcal X/T)-2)+(wp+v)(p-1)
$$
implies $g(\mathcal X/T)=0$, and hence
$$
p^2-3\ge 2g(\cX)-2=-2p+(wp+v)(p-1),
$$
which can be read
$$
p^2-1\ge 2(-p+1)+(wp+v)(p-1),
$$
and hence
$$
p+1\ge wp+v-2.
$$

Then $w=1$ and either $v=2$ or $v=3$.
Actually, the latter case cannot occur.
Indeed, the Riemann-Hurwitz formula for $G$ gives
$$
p^2-3\ge 2g(\cX)-2= -2p^2+3(p^2-1)+p(p-1)
$$
and hence $p(p-1)\le 0$, a contradiction.
\end{proof}

\begin{thm}\label{24mar} If $N=p^h\ge 2g(\cX)+1$, then $h=1$ and $\mathcal X:y^2=x^p-x$.
\end{thm}
\begin{proof}
By Proposition \ref{10mar}, $h\le 2$ holds.
Assume by contradiction that $h=2$. By Proposition \ref{13mar},
the subgroup of $G$ of order $p$ fixes a point $P_0$, coincides with $G_{P_0}^{(2)}$, and acts semiregularly on the other points of $\mathcal X$. Also, the quotient curve $\mathcal X/G_{P_0}^{(2)}$ is rational.

By \cite[Theorem 12.5]{HKT}, an equation for $\mathcal X$ is 
$$
y^p-y=B(x)
$$
with $B$ a polynomial of degree $d$ coprime to $p$. 
By \cite[Lemma 12.1(b)]{HKT} the genus of $\mathcal X$ is $(p-1)(d-1)/2$.
Then, by Proposition \ref{13mar}, $d=p+1$ holds. 

This means that $\mathcal X$ is a non-singular plane curve and therefore the group $G$ consists of linear transformations fixing the points of $\mathcal X$; see \cite[Theorem 11.29]{HKT}. Clearly $G$ must fix the point at infinity of the $Y$-axis and the tangent line of $\mathcal X$ at that point, namely the ideal line.
It was noticed in \cite{ABFG}
that  for a $p$-element in the projective linear group $\mathrm{PGL}(3,\mathbb K)$ represented by a matrix
$$
A_{a,b,c}=\begin{pmatrix}
1 & 0 & 0 \\
b & 1 & 0\\
c & a & 1
\end{pmatrix},
$$
with $a,b,c \in \mathbb K$, the equation 
$$
A^p=A_{pa,pb,p\frac{p-1}{2}ab+pc}=I_3\,\, ,
$$
holds for $p>2$. Then, there cannot exist an element of $G$ of order $p^2$.
Thus $h=1$ and by \cite{Homma} an equation of $\mathcal X$ is $y^2=x^p-x$.
\end{proof}

\begin{rem} In \cite{MR} it is shown that $2g(\mathcal X)p/(p-1)$
is an upper bound for the size of an {\em abelian} $p$-group of automorphisms of a curve $\mathcal X$. 
By the result of the present section, the bound can be improved in the cyclic case whenever $g(\mathcal X)>\frac{p-1}{2}$.
\end{rem}

\subsection{The case $N=p^hm$ with $m>1$ and $(p,m)=1$}

    Let $M$ be the subgroup of $G$ of order $m$ and let $S$ be the subgroup of $G$ of order $p^h$. Also, let $H$ be the factor group $G/M$. Then $H$ is a cyclic group of order $p^h$ acting on the quotient curve $\mathcal X/M$.

\begin{prop}\label{5apr} 
Assume that $N=p^hm\ge 2g(\cX)+1$ with $p\ge 3$, $m>1$, and $(p,m)=1$. 
If $\mathcal X/M$ is elliptic, then $p=3$, $h=1$, the $p$-rank of $\mathcal X/M$ is zero, and $H$ does not act semiregularly on $\mathcal X/M$.
\end{prop}
\begin{proof}
By Proposition \ref{ellittica} it is enough to show that $\mathcal X/M$  is not ordinary.
Assume by contradiction that the $p$-rank of $\mathcal X/M$ is equal to 1. Then by Proposition \ref{ellittica}, $H$ acts semiregularly on $\mathcal X/M$. Therefore, also $\mathcal X/G$ is elliptic and the Riemann-Hurwitz genus formula applied to $G$ gives
$$
2g(\cX)-2=\sum_{i=1}^r(N-\ell_i),
$$
where $r$ is the number of short orbits under $G$ and $\ell_i$ denote their sizes. As $N-\ell_i\ge N/2$ we have that if $r\ge 2$ then
$$
2g(\cX)-2\ge N,
$$
a contradiction. Therefore, $r=1$ holds. Moreover, no $p$-element of $G$ fixes a point in $\mathcal X$ and therefore $\ell_1=p^hm'$ with $m'<m$ a divisor of $m$. 

Note that  $\mathcal X/S$ is a cyclic extension of $\mathcal X/G$ of order $m$, and hence it is a Kummer extension with Galois group $G/S$. However, there is only one short orbit of $\mathcal X/S$ under the action of $G/S$. This is impossible by Proposition \ref{n=1}.
\end{proof}

\begin{prop}\label{10marbis} 
Assume that $N=p^hm$ with $p\ge 3$, $m>1$, and $(p,m)=1$. 
If $N\ge 2g(\cX)+1$, then $h\le 2$.
\end{prop}

\begin{proof}
    Recall that  $H$ is a cyclic group of order $p^h$ acting on the quotient curve $\mathcal X/M$. 
    Let $\hat g$ be the genus of $\mathcal X/M$.
    Three cases are distinguished.
\begin{itemize}
    \item $\hat g=0$. By the classification of subgroups of $\mathrm{PGL}(2,\mathbb K)$ it follows that $h=1$ holds in this case. 
    
    \item $\hat g=1$. By  Proposition \ref{5apr} we have  $h=1$.

    \item $\hat g\ge 2$. By the Riemann-Hurwitz genus formula applied to $M$ we have
    $$
    2g(\cX)-2\ge m(2\hat g-2),
    $$
    that is
    $$
    (2\hat g-2)\le \frac{2g(\cX)-2}{m}= \frac{2g(\cX)+1-3}{m}\le p^h-\frac{3}{m},
    $$
    and hence
    $$
    |H|\ge 2 \hat g(\cX)-2+\frac{3}{m}
    $$
    which implies $$|H|\ge 2 \hat g-1;$$ finally, Proposition \ref{10mar} applies.
\end{itemize}
    \end{proof}

\begin{prop}\label{12mar} 
Assume that $N=p^hm\ge 2g(\cX)+1$ with $p\ge 3$, $m>1$, and $(p,m)=1$. Let $S$ be the subgroup of $G$ of order $p^h$. Then $h=1$ and  one of the following holds:
\begin{itemize}
\item[(i)] $\gamma(\mathcal X/M)=0$  and $g(\mathcal X/S)\leq 1$; also, the only non-tame orbit $\Omega$ under the action of $G$ has size $|\Omega|\le 3$ for $g(\mathcal X/S)=1$ and $|\Omega|\le 2$ for $g(\mathcal X/S)=0$. If $g(\mathcal X/S)=1$, then $m\in \{2,3,4,6\}$. 


\item[(ii)] $\gamma(\mathcal X/M)=g(\mathcal{X}/M)=2$ and $p=3$; the group $G/M$ of order $3$ fixes two points on $\mathcal X/M$ and acts semiregularly on the other points.
Also, up to birational equivalence $\mathcal{X}/M$ has affine equation $$by^3+cy = ax +
1/x,$$ where $a,b,c \in \mathbb{K}^*$, $H$ consists of the maps
$$
y\mapsto y+v, \qquad x\mapsto x,
$$
where $v$ runs over the roots of  $bY^3+cY$, and $|\aut(\mathcal X/M)|=12$.

\end{itemize}
\end{prop}
\begin{proof}
By the proof of Proposition \ref{10marbis}, the quotient curve $\mathcal X'=\mathcal X/M$ of genus $\hat g$ admits a cyclic automorphism group $H$ of size $p^h$ with  $p^h\ge 2\hat g-1$.

Assume that the $p$-rank of $\mathcal X'$ is greater than zero. 
From $p^h\ge 2\hat g-1$ it follows by \cite[Theorem 1]{NAKA} that either $\hat g=1$ or 
$\hat g\ge 2$ and $p=3$. The former case is ruled out by Proposition \ref{5apr}.
In the latter case, it has already been noticed in the proof of Proposition \ref{10mar} that  $p^h=3$ holds and therefore $\hat g=2$. Then, (ii) holds by Proposition \ref{MZ}.

Assume now that the $p$-rank of $\mathcal{X}'$ is zero. Then $H$ fixes one point $\bar P$ of $\mathcal X'$ and acts semiregularly on the other points.
Let $\Omega$ be the orbit of $\mathcal X$ under the action of $M$ corresponding to the point $\bar P$ of $\mathcal X'$. Clearly $\Omega$ is the only non-tame orbit under the action of $G$.

Moreover, $\Omega$ is both an $M$-orbit and a $G$-orbit, as $s(\bar P)=\bar P$ for each $s\in S$. This means that for each point $P\in \Omega$ we have
$$
|\Omega|=\frac{p^hm}{|G_P|}=\frac{m}{|M_P|}
$$
which implies that $p^h\mid |G_P|$ and $p^h=|G_P^{(1)}|$, \textcolor{black}{that is, $G_P=H\times M_P$}. Let $m'=\frac{m}{|\Omega|}=|M_P|$. 

 We apply the Riemann-Hurwitz genus formula to $H$ and we denote by $\tilde g$ the genus of $\mathcal X/S$. 
 
 First, we show that the case $h=2$ cannot actually occur. Assume that $|S|=|H|=p^2$.
 Since $G_P$ is abelian and $G_P/G_P^{(1)}$ has order $m'$, by \cite[Lemma 11.75 (iv)]{HKT} we have that there exist positive integers $v$ and $w$ such that
$$
S=G_P^{(1)}=G_P^{(2)}=\ldots=G_P^{(vm')},
$$
and  $$|G_P^{(vm'+1)}|=\ldots=|G_P^{((v+w)m')}|=p,$$
and $|G_P^{((v+w)m'+1)}|=1$.
 
Taking into account \cite[Lemma 11.75 (v)]{HKT}, we have
$$
vm'\equiv (v+w)m' \pmod p
$$
and since $(m',p)=1$ we have that $p$ divides $w$.
Therefore, $w\ge p$ yields
$$
2g(\cX)-2\ge p^2(2\tilde g-2)+|\Omega|(m'p^2-m'+p^2-1+p^2m'-pm')
$$
and hence
$$
p^2m-3\ge-2p^2+ 2mp^2-m(p+1)+|\Omega|(p^2-1)
$$
whence
$$
2p^2-3\ge m(p^2-p-1)+|\Omega|(p^2-1)\ge 2(p^2-p-1)+p^2-1
$$
and
$$
 p^2-2p\le 0
$$
a contradiction.

Therefore, $h=1$ can be assumed. By the Riemann-Hurwitz genus formula applied to $S$, taking into account (iv) of \cite[Lemma 11.75]{HKT} we have 
$$
mp-3\ge 2g(\cX)-2\ge p(2\tilde g-2)+|\Omega|((m'+1)(p-1))
$$
and hence
$$
m\ge p(2\tilde g-2)+3+|\Omega|(p-1).
$$
Note that this is a contradiction with the bound on the size of an abelian subgroup of the automorphism group of $\mathcal{X}/S$ (see \cite[Theorem 11.79]{HKT}), unless $\tilde g\leq 1$ or $p=3$,
$\tilde g=2$, and $|\Omega|=1$. However, by \cite[Theorem 11.60]{HKT}, the latter case cannot actually occur taking into account that $M$ fixes a point.

Therefore, either $\tilde g=0$ or $\tilde g=1$.

Observe that $m>|\Omega|$ holds with the only  exception $m=|\Omega|=2$ for $\tilde g=0$.

To investigate $|\Omega|$  we  first assume that $\tilde g=0$. Then $G/S$ is a cyclic group of order $m$ coprime with $p$ acting on the projective line on an algebraically closed field of characteristic $p$. Therefore, if a non-trivial element of $G/S$ fixes a point, then it has precisely two fixed points on $\mathcal X/S$ and acts semiregularly on the other points. It has already been noticed that either $m>|\Omega|$ or $m=|\Omega|=2$. Assume that $m>|\Omega|$. As each point in $\Omega$ is fixed by $S$, there exist $|\Omega|$ points on the projective line fixed by some elements in $G/S$; therefore $|\Omega|\le 2$.  

If $\tilde g=1$, then $m>|\Omega|$ and hence $G/S$ admits a non-trivial (cyclic) subgroup of size $m'=m/|\Omega|>1$ fixing each point of $\mathcal X/S$ corresponding to a point in $\Omega$.  
Then $G/S$ does not act semiregularly on $\mathcal X/S$ and by Proposition \ref{ellittica} 
we have $m\in \{2,3,4,6\}$. As $m=m' |\Omega|$ one of the following holds:
\begin{itemize}
\item[(e1)] $m=m'=2$, $|\Omega|=1$;

\item[(e2)] $m=m'=3$, $|\Omega|=1$;

\item[(e3)] $m=m'=4$, $|\Omega|=1$;

\item[(e4)] $m=4, m'=2$, $|\Omega|=2$;

\item[(e5)] $m=m'=6$, $|\Omega|=1$;

\item[(e6)] $m=6, m'=2$, $|\Omega|=3$;

\item[(e7)] $m=6, m'=3$, $|\Omega|=2$.

\end{itemize}
\end{proof}

\begin{prop}\label{15mar}
    Assume that $N=pm\ge 2g(\cX)+1$ with $p\ge 3$, $m>1$, and $(p,m)=1$. Let $S$ be the subgroup of $G$ of order $p$ and $M$ the subgroup of $G$ of order $m$.
If $\mathcal X/M$ is not rational then  $p=3$ and $g(\mathcal X/M)\in \{1,2\}$.

\end{prop}
\begin{proof}
Let $\hat g$ denote $g(\mathcal X/M)$ and let $H=G/M$. By Proposition \ref{12mar}, either $\mathcal{X}/M$ has $p$-rank equal to zero, or $p=3$, $\hat g=2$, and we are in case (ii) of Proposition \ref{12mar}. Assume that the former case holds, and as in the proof of Proposition \ref{12mar} let $\Omega$ be the only non-tame orbit under the action of $G$. First, we prove that the extension $\mathcal X$ over $\mathcal X/M$ ramifies. Assume by contradiction that the extension is unramified; then, in particular, $m=|\Omega|$ holds. From the proof of Proposition \ref{12mar}, we have that $m=2$ and $\mathcal X/S$ is rational. Also,
$$
2p-3\ge 2g(\cX)-2=4\hat g-4.
$$
Then $p\ge 2\hat g-\frac{1}{2}$.
As $g(\cX)\ge 2$ we have that $\hat g\ge 2$ as well.
By \cite{Homma}, $p= 2\hat g+1$ and hence
$2g(\cX)-2 = 2(p-3)$. This yields $g(\cX)+2=p$ which is impossible by \cite{Homma}.

Since the extension is a Kummer extension, by Proposition \ref{n=1} it is not possible that $\Omega$ is the only short orbit  of $\mathcal X$ under the action of $M$. 
Therefore, there exists a short orbit  $\Omega'\neq \Omega$ under the action of $M$. 
If $\hat g>2$ or $\hat g=2$ and $p>3$ then $\Omega'$ is not fixed by $S$ (see the proof of Proposition \ref{12mar}). A proper subgroup $M'$ of $M$ fixes each point in $\Omega'$. On the other hand, $S$ acts on the set of fixed points by $M'$ semiregularly, since $S(\Omega')\ne \Omega'$. Therefore, $M$ has at least $p$ short orbits of size $\Omega'$, other than $\Omega$. 

Hence, 
$$
pm-3\ge 2g(\cX)-2\ge m(2\hat g-2)+m-|\Omega|+p(m-|\Omega'|).
$$
Therefore,
$$
p\ge \frac{m}{|\Omega'|}(2\hat g-2)+\frac{m+3-|\Omega|}{|\Omega'|}
$$
Note that $\frac{m+3-|\Omega|}{|\Omega'|}$ is strictly greater than $1$ as, for $m>|\Omega|$ we have
$$
\frac{m+3-|\Omega|}{|\Omega'|}\ge \frac{m+3-\frac{m}{2}}{\frac{m}{2}},
$$
whereas $m=|\Omega|$ yields $m=2$ and hence
$\frac{m+3-|\Omega|}{|\Omega'|}=3$.

Then,
$$
p\ge \frac{m}{|\Omega'|}(2\hat g-2)+2\ge 4\hat g-2>2\hat g+1.
$$
By \cite{Homma}, $\hat g\le 1$ holds and by Proposition \ref{5apr} we have that $\hat g=1$ implies $p=3$.
\end{proof}

\begin{prop}\label{1mag} Let $p\geq 3$. Let $G$ be a cyclic group of automorphisms acting on a curve $\cX$ of genus $g(\cX)\ge 2$ of order $N$ greater than or equal to $2g(\cX)+1$. If $|G|=p^hm$ with $h\ge 1$, $m>1$, and $(p,m)=1$ then 
 $|G|=pm$. Also, for $S$ the subgroup of $G$ of order $p$ and $M$ the subrgoup of $G$ of order $m$, one of the following holds:
\begin{itemize}
\item[(Ia)] $g(\mathcal X/M)=g(\mathcal X/S)=0$, $g(\cX)=\frac{1}{2}(p-1)(m-1)$, and $S$ fixes exactly one point;
\item[(Ib)] $g(\mathcal X/M)=g(\mathcal X/S)=0$, $m=2$, $g(\cX)=p-1$, and $S$ fixes exactly two points.

\end{itemize}

\end{prop}
\begin{proof}
 Recall that $G/S$ is cyclic group of order $m=N/p$ acting on $\mathcal X/S$. Also, points of $\mathcal X/S$ are either fully ramified or unramified with respect to the covering $\mathcal X \to \mathcal X/S$.

 We first prove that the $p$-rank $\gamma(\mathcal X/M)$ of $\mathcal{X}/M$ is equal to $0$. Assume on the contrary that 
$\gamma(\mathcal X/M)>0$. Then, by Proposition \ref{12mar}, we are in the following setting: $p=3$ and $\mathcal X/M$ \textcolor{black}{is an ordinary curve of} genus $2$; the group $G/M$ of order $3$ fixes two points on $\mathcal X/M$ and acts semiregularly on the other points. 

We are going to show that actually such situation cannot occur.
     As $g(\cX/M)=2$ and $3m\geq 2g(\cX)+1$, the Riemann-Hurwitz formula applied to $M$ gives
    $$
    3m-3\geq 2g(\cX)-2=2m+\sum_{i=1}^n (m-|\Omega_i|)\ge 2m+n\frac{m}{2},
    $$
    where $\Omega_1,\ldots, \Omega_n$ are the short orbits of $M$ on $\cX$. Therefore, $n\leq 1$ holds. However, by Proposition \ref{n=1}, we have $n=0$, and hence $g(\cX)=m+1$. Now, let $\bar{P}_1, \bar{P}_2$ be the two points on $\cX/M$ fixed by $G/M$. Then, $S$ fixes each point from the two (long) $M$-orbits, say $\Omega_1$ and $\Omega_2$,  lying over $\bar{P}_1$ and $\bar{P}_2$. This means that $S$ fixes at least $2m$ points of $\cX$. However, by \cite[Lemma 11.106]{HKT} the maximum number of fixed points of $S$ is
    $$
    2+\frac{2g(\cX)}{|S|-1}=m+3,
    $$
   which yields $m=2$. However, for $m=2$,  $p=3$, and $g(\cX)=3$, we have that $pm\geq 2g(\cX)+1$ does not hold.

 So, from now on we assume that the $p$-rank of $\mathcal{X}/M$ is zero. Then, by Propositions \ref{5apr}, \ref{12mar} and \ref{15mar} we have the following information:
 \begin{itemize}
     \item $g(\mathcal X/M)\leq 1$. Moreover, $g(\mathcal X/M)=1$ implies $p=3$;
     \item $g(\mathcal X/S)\leq 1$;
     \item there exists a unique non-tame orbit $\Omega$ under the action of $G$ on $\cX$.
 \end{itemize}
Clearly, fully ramified points of $\mathcal X/S$ correspond to the points of $\Omega$, and with a slight abuse of notation they will be identified with the points of $\Omega$.
If a point $\bar P\in \mathcal X/S$ not in $\Omega$ is fixed by an element $\bar \alpha$ in $G/S$, then the element of $M$ of the same order, say $\alpha$, fixes all the $p$ points over $\bar P$ with respect to the covering $\mathcal X \to \mathcal X/S$. On the other hand, if $\bar P\in \Omega$ then $\alpha$ fixes the only point over $\bar P$.
Now, let $v$ be the number of points fixed by $M$ in $\mathcal X$. We distinguish two cases according to whether $g(\mathcal X/S)=0$ or $g(\mathcal X/S)=1$.

\textbf{CASE 1: $g(\mathcal X/S)=0$.}

In this case, either $G/S$ doesn't fix any point of $\mathcal X/S$, or it fixes two points and acts semiregularly on the other points of $\mathcal X/S$.
Note also that in the latter case these two points cannot be both in $\Omega$, as $\Omega$ is an $M$-orbit (see the proof of Proposition \ref{12mar}).
Taking into account that when $g(\mathcal X/S)=0$ then $|\Omega|\le 2$, by the discussion above we have
$$
v\in\{0,p+1,2p\},
$$
and $v=2p$ can occur only when $|\Omega|=m=2$.

We now show that $g(\mathcal X/M)=0$. If on the contrary  $g(\mathcal X/M)=1$, then $p=3$, and by the Riemann-Hurwitz formula applied to $M$ we have 
$$
2g(\cX)-2=v(m-1)\in \{0,4(m-1),6\},
$$
a contradiction.

Then $g(\mathcal X/M)=0$ and by the Riemann-Hurwitz formula applied to $M$ we have 
$$
2g(\cX)-2=-2m+v(m-1)=(v-2)m-v;
$$
hence,
$$
2g(\cX)-2\in \{-2m,(p-1)(m-1)-2,2p-4 \}.
$$
As $2g(\cX)-2=-2m$ cannot occur, we have that either (Ia) or (Ib) holds.

\textbf{CASE 2: $g(\mathcal X/S)=1$.}

By Proposition \ref{12mar} and its proof, we know that $m>|\Omega|$, $G/S$ admits a (non-trivial) cyclic subgroup of order $m'=m/|\Omega|$ fixing each point of $\cX/S$ corresponding to a point in $\Omega$, and one of the cases (e1)-(e7) holds. We will prove that each of these cases cannot actually occur. To do so, first note that by the Riemann-Hurwitz formula we have
\begin{equation}\label{RHfixed}
2g(\cX)-2=m(2g(\mathcal X/M)-2)+\sum_{\alpha \in M, \alpha\neq id}N(\alpha)\geq -2m+ \sum_{\alpha \in M, \alpha\neq id}N(\alpha),
\end{equation}
where $N(\alpha)$ is the number of points of $\mathcal X$ fixed by $\alpha$. We show that provides a contradiction with $N=pm\geq 2g(\cX)+1$ in each of the cases (e1)-(e7).
\begin{itemize}
\item[(e1)]  $m=2$, $|\Omega|=1$, $G/S$ fixes $4$ points, and one of such points is in $\Omega$. Then, Equation \eqref{RHfixed} yields 
$$
2g(\cX)-2\ge -2(2)+1+3p,
$$
a contradiction.
\item[(e2)] $m=3$, $|\Omega|=1$, $G/S$ fixes $3$ points, and one of such points is in $\Omega$. Then, Equation \eqref{RHfixed} yields 
$$
2g(\cX)-2\ge -2(3)+2(1+2p),
$$
a contradiction.
\item[(e3)] $m=4$, $|\Omega|=1$, $G/S$ fixes the point in $\Omega$ and one point not in $\Omega$, whereas its subgroup of order $2$ fixes the point in $\Omega$ and two points not in $\Omega$. Then, Equation \eqref{RHfixed} yields
$$
2g(\cX)-2\ge -2(4)+(p+1+p+1+3p+1),
$$
a contradiction.
\item[(e4)] $m=4$, $|\Omega|=2$, $G/S$ fixes two points not in $\Omega$, whereas its subgroup of order $2$ fixes two points in $\Omega$ and two points not in $\Omega$.
Then, Equation \eqref{RHfixed} yields
$$
2g(\cX)-2\ge -2(4)+(2p+2p+2p+2),
$$
a contradiction.
\item[(e5)] $m=6$, $|\Omega|=1$, 
 $G/S$ fixes the point in $\Omega$, its subgroup of order $3$ fixes two more points not in $\Omega$, and its subgroup of order $2$ fixes $3$ more points not in $\Omega$.
Then, Equation \eqref{RHfixed} yields
$$
2g(\cX)-2\ge -2(6)+(1+1+1+2p+1+2p+1+3p),
$$
a contradiction, unless $p=3$. However, in this case $p=3$ is not possible as $p=3$ divides $m=6$.
\item[(e6)] $m=6$, $|\Omega|=3$, 
 $G/S$ fixes one point not in $\Omega$, its subgroup of order $3$ fixes three points not in $\Omega$, and its subgroup of order $2$ fixes the three points in $\Omega$ and one point not in $\Omega$.
Then, Equation \eqref{RHfixed} yields
$$
2g(\cX)-2\ge -2(6)+(p+p+3p+3p+3+p),
$$
a contradiction.
\item[(e7)] $m=6$, $|\Omega|=2$, 
 $G/S$ fixes one point not in $\Omega$, its subgroup of order $3$ fixes two points in $\Omega$ and one point not in $\Omega$, and its subgroup of order $2$ fixes four points not in $\Omega$.
Then, Equation \eqref{RHfixed} yields
$$
2g(\cX)-2\ge -2(6)+(p+p+p+2+p+2+4p),
$$
a contradiction.
\end{itemize}
\end{proof}

\begin{prop}
    In Case (Ia) of Proposition \ref{1mag}, up to birational equivalence $\cX$ is the curve of affine equation
    $$
 y^p-y=a(x^m-b),
$$
with $a,b\in \mathbb{K}$.
\end{prop}
\begin{proof}
By \cite[Theorem 12.5]{HKT},
an equation for $\mathcal X$ is
$$
y^p-y=B(x),
$$
where $B$ is a polynomial of degree $d$. Since the genus of this curve is $(p-1)(d-1)/2$ we have that $m=d$. 
Also, the fixed field of $S$ is $\mathbb K(x)$. Two orbits under the action of $S$ are fixed by $G/S$, one of size $p$, say $\Omega_1$,  and the other consisting of the single point in $\Omega$. Clearly the only pole of $x$ in $\mathcal X/S$, namely the ideal point of the projective line, corresponds to $\Omega$. Let $P_1$ be the point lying under $\Omega_1$ and let $a\in \mathbb K$ coincide with $x(P_1)$.
Without loss of generality we can assume that $a=0$. Then the divisor of $x$ on $\mathcal X$ 
is preserved by every element in $G$.

Let $\sigma$ be a generator of $G$. Then
$\div(\sigma^*(x))=\div(x)$ and therefore, $\sigma^*(x)=\mu x$ for some $\mu \in \mathbb K$. By $\sigma^{pm}=1$ we deduce that $\mu^m=1$.

Consider  $(\sigma^p)^*(x)$; clearly, $(\sigma^p)^*(x)=\nu x$ holds for some  $m$-th root of unity $\nu$. We claim that $\nu $ is primitive. In fact, if $\nu^h=1$ for a proper divisor of $m$, then $x$ would be fixed by  $\sigma^{ph}$, which is not in $H$ and this would contradict that $\mathbb K(x)$ is the fixed field of $S$.
As $x^m$ is fixed by $G$, and the degree of the extension $\mathbb K(\mathcal X)$ over $K(x^m)$ coincides with $pm=|G|$, we have that $\mathbb K(x^m)$ is the fixed field of $G$. 
Note that $\mathbb K(y)$ and the fixed field of $M$ are both intermediate fields between $\mathbb K(\mathcal X)$ and $\mathbb K(x^m)$ such that the degrees  $\mathbb K(\mathcal X):\mathbb K(y)$ and 
$\mathbb K(\mathcal X):\text{Fix}(M)$ are both equal to $m$. Since $G$ has a unique subgroup of order $m$, we have that $\mathbb K(y)$ is the fixed field of $M$ and hence $M$ fixes $y$. 

Therefore, taking into account that $M=\langle \sigma^p\rangle$, we have that for each $j=1,\ldots, m-1$,
$$
(\sigma^{jp})^*(y^p)-(\sigma^{jp})^*(y)=B((\sigma^{jp})^*(x)),
$$
which implies
$$
y^p-y=B(\nu^j x),
$$
and hence
$$
B(x)=B(\nu^j x)
$$
holds for each $j=1,\ldots, m-1$.

Since $\nu$ is a primitive $m$-th root of unity, we have  $B(x)=a(x^m-z_0)$ for some $z_0\in \mathbb K$.
\end{proof}

\begin{prop}
    In Case (Ib) of Proposition \ref{1mag}, up to birational equivalence $\cX$ is the curve of affine equation
    $$
 by^p+cy=ax+\frac{1}{x}
$$
with $a,b,c\in \mathbb{K}$.
\end{prop}
\begin{proof}
For $m=2$ and $g(\cX)=p- 1$ Proposition \ref{MZ} applies. Then an affine equation for $\mathcal X$ is either
$$
by^p+cy=ax+\frac{1}{x},
$$
or 
$$
by^p+cy=x^3+bx.
$$
In the former case, clearly $\mathcal X$ admits a cyclic automorphism group of size $2p$ consisting of the maps
$$
y\mapsto y+\gamma, \qquad x\mapsto 1/(ax),
$$
and
$$
y\mapsto y+\gamma, \qquad x\mapsto x,
$$
where $\gamma$ runs over the roots of  $bY^p+cY$.
In the latter case, 
the group of order $p$ fixes only one point, contradicting $|\Omega|=2$.
\end{proof}


The following statement summarizes the results obtained in this subsection. 
\begin{thm}\label{24marBIS} 
If $|G|=p^hm$ with $h\ge 1$ and $(p,m)=1$ then  $|G|=pm$ with $m>1$, $p\neq 3$ and
$$
g(\cX)\in\{(p-1)(m-1)/2,p-1\};
$$
also, $\mathcal X$ is an Artin-Schreier extension of the projective line. Specifically, up to birational equivalence, if $g(\cX)=(p-1)(m-1)/2$, then
$$
\mathcal X: y^p-y=a(x^m-b),
$$
with $a,b\in \mathbb K$, whereas if 
$g(\cX)=p-1$ then
$$
\mathcal X: by^p+cy=ax+\frac{1}{x},
$$
with $a,b,c\in \mathbb K$.

\end{thm} 

\section{Proof of Theorem \ref{MAIN} when $p$ does not divide $N$}\label{coprimo}

In  \cite{Sasaki} a classification of curves admitting a large cyclic automorphism group was given for the case $\mathbb K=\mathbb C$. Here we prove that an analogue result holds when the characteristic $p$ does not divide the size of the cyclic group. 
Most of the steps in the proof presented in \cite{Sasaki} hold with little modifications; however, some exceptions arise for implications that depend on topological results (see, for example, theorems from references [2] and [3] in \cite{Sasaki}).

Throughout the Section, $p>2$ and $(p,N)=1$ are assumed. We will show that if $N\ge 2g(\mathcal X)+1$ then   (I) in Theorem \ref{MAIN} holds.

 Let $G=\langle\sigma\rangle$ and $\left\{\lambda_1,\dots,\lambda_n\right\}$ be the set of points of $\mathcal{X}/G$ where the covering $\pi:\mathcal{X}\longrightarrow \mathcal{X}/G$ ramifies. Using the same terminology and notation as in \cite{Sasaki}, the automorphism $\sigma$ is said to be of kind $(g_0;e_1,e_2,\dots,e_n)$ if the genus of $\mathcal{X}/G$ is $g_0$ and the ramification index at $P_i$ is $e_i$, where $P_i$ is any point in $\mathcal{X}$ such that $\pi(P_i)=\lambda_i$.

\begin{prop}\label{mcm}
Let $M$ be the l.c.m. of $\left\{e_1,\dots,e_n\right\}$. Then the following hold:
\begin{itemize}
    \item[i)] $n\geq 2$ and if $g_0=0$ then $n\geq 3$;
    \item[ii)] $M$ divides $N$, and, if $g_0=0$, $M=N$;
    \item[iii)] the $l.c.m.$ of $\left\{e_1,\dots,\hat{e}_i,\dots e_n\right\}$ is $M$ for each $i$, where $\hat{e}_i$ denotes the omission of $e_i$.
\end{itemize}
\end{prop}
\begin{proof}\begin{itemize}
\item[i)] By Proposition \ref{n=1} and Proposition \ref{n=0} $n\geq 2$ holds. Assuming $g_0=0$ and $n=2$, the Riemann-Hurwitz formula gives
$$2g(\cX)-2=-2N+\dfrac{N}{e_1}(e_1-1)+\dfrac{N}{e_2}(e_2-1)$$
that is
$$2g(\cX)-2=-\dfrac{N}{e_1}-\dfrac{N}{e_2},$$
 which contradicts $g(\cX)\geq 2$.
\item[ii)] As each $e_i$ divides $N$, the same holds for $M$.
Moreover, by Proposition \ref{divisibilità}, the covering $\cX/\langle\sigma^{N/M}\rangle\longrightarrow \cX/G$ is unramified, and hence the Riemann-Hurwitz formula for the covering $\cX/\langle\sigma^{N/M}\rangle\longrightarrow \cX/G$, together with $g_0=0$, reads
$$2\bar{g}-2=-2\dfrac{N}{M},$$
where $\bar{g}$ is the genus of $\mathcal{X}/\langle\sigma^{N/M}\rangle$.
This implies $\bar{g}=0$ and $N=M$.
\item[iii)] Let $M_i$ be the $l.c.m.$ of $\left\{e_1,\dots,\hat{e}_i,\dots,e_n\right\}.$
If there exists $j\neq i$ such that $e_j=e_i$ then clearly $M_i=M$. So, assume $e_j\neq e_i$ for each $j\neq i$.
Then, by Proposition \ref{divisibilità} with $d=M_i$, the covering $\mathcal{X}/\langle\sigma^{N/M_i}\rangle\longrightarrow X/G$ can be ramified only over $\lambda_i$. As this is not possible by Proposition \ref{n=1}, $\mathcal{X}/\langle\sigma^{N/M_i}\rangle$ over $X/G$ is unramified and hence $e_i$ divides $M_i$, which concludes the proof.
\end{itemize}
\end{proof}

Throughout this section, we denote by $A_N$ the set
$$
A_N = \left\{(r, s) \in \mathbb{Z}^2 : r, s \geq 1 \text{ and } r + s \leq N - 1\right\},
$$
and, according to \cite{Sasaki}, we refer to an element $(r, s) \in A_N$ as a \textit{primitive pair} if $\gcd(r, s, N) = 1$.\\

We denote by $F(r, s)$ the non-singular model of equation $y^N = x^r (1 - x)^s$, with function field $\mathbb{K}(x, y)$ and genus $g(r, s)$. 
Let $\sigma(r, s)$ be the automorphism of $F(r, s)$ defined by
$$
\sigma(r, s): \begin{cases}
    x' = x, \\
    y' = \zeta _N y,
\end{cases}
$$
where $\zeta _N$ is a primitive $N$-th root of unity. Since $p \nmid N$, $F(r, s)$ is a Kummer extension of the projective line $\mathbb P^1(\mathbb K)$. We denote by
$$
\pi(r, s): F(r, s) \longrightarrow \mathbb{P}^1(\mathbb K)
$$
the morphism induced by $\mathbb{K}(x) \subset \mathbb{K}(x, y)$.
As $(r, s)$ is a primitive pair, Kummer theory implies that $\pi(r, s)$ ramifies exactly at the three points $(0 : 1)$, $(1 : 1)$, and $(1 : 0)$.

\begin{thm}\label{uguaglianza}
If $(r, s) \in A_N$ is a primitive pair, then $N\ge 2g(r,s)+1$.
\end{thm}

\begin{proof} 
From Kummer theory and the Riemann-Hurwitz formula, we have
\begin{eqnarray*}
2g(r, s) - 2 &=& -2N + \gcd(N, r) \left(\frac{N}{\gcd(N, r)} - 1\right) + \gcd(N, s) \left(\frac{N}{\gcd(N, s)} - 1\right) + \\
&& + \gcd(N, r+s) \left(\frac{N}{\gcd(N, r+s)} - 1\right),
\end{eqnarray*}
which simplifies to
$$
g(r, s) = \frac{1}{2}\left(N + 2 - \gcd(N, r) - \gcd(N, s) - \gcd(N, r+s)\right).
$$
This clearly implies that
$$
N \geq 2g(r, s) + 1.
$$
\end{proof}
Another family of curves will play an important role in this section. For $\lambda \in \mathbb{K} \setminus \{0, 1\}$, let $H_{\lambda}$ be the hyperelliptic curve of genus $g$ defined by the equation
$
y^2 = (x^{g+1} - 1)(x^{g+1} - \lambda)
$
and let $\tau_{\lambda}$ be an automorphism of $H_{\lambda}$ defined by
$$
\tau_{\lambda}^{*} : (x, y) \longmapsto (\zeta _{g+1} x, -y),
$$
where $\zeta _{g+1}$ is a primitive $(g+1)$-th root of unity.

\begin{thm}\label{isomorfe}
Let $\mathcal{X}$ be a curve of genus $g(\cX)\geq 2$ admitting an automorphism of $\mathcal{X}$ of order $N\geq 2g(\cX)+1$, with $(p,N)=1$ . Then $\mathcal{X}$ is isomorphic to either $F(r,s)$ for some primitive pair $(r,s)\in A_N$ or to $H_{\lambda}$ for some $\lambda\in\mathbb{K}\setminus\left\{0,1\right\}$ with $N=2g(\cX)+2$ and $g(\cX)$ even. 
\end{thm}
\begin{proof}
Let $\sigma$ be an automorphism of $\cX$ of order $N$. Let  $(g_0;e_1,e_2,\dots,e_n)$ be its type. We can assume that $e_1\leq e_2\leq\cdots\leq e_n$ holds. By the Riemann-Hurwitz formula applied to $G=\langle \sigma\rangle$ we have
\begin{equation}\label{RH}
\dfrac{2g(\cX)-2}{N}=2g_0-2+\sum_{i=1}^{n}\left(1-\dfrac{1}{e_i}\right).
\end{equation}
We first prove the following three claims.

\begin{enumerate}
    \item { ${\mathbf{g_0=0}}$}. As $N\geq 2g(\cX)+1$,  the left side of \eqref{RH} is smaller than $1$. Moreover, by Propositions \ref{n=0} and  \ref{n=1}, $n\geq 2$. This, together with $g_0\geq 1$, would imply that the right side of \eqref{RH}) is greater then or equal to $1$, a contradiction. So, $g_0=0$ holds.
    \item {\bf If $N$ is odd, then ${\mathbf{n=3}}$}. Since $g_0=0$, we have $n\geq 3$ by Proposition \ref{mcm}. Observe that as $N$ is odd, each $e_{i}$ is odd. This together with \eqref{RH} yields $n<5$. We will now rule out the case $n=4$. The possibilities for $n=4$ are:
    \begin{itemize}
        \item $n=4$ and $e_1\geq 5$. In this case the right-hand side of \eqref{RH} is greater than $1$, a contradiction.
        \item $n=4$, $e_1=3$ and $e_2\geq 5$. As above, the right-hand side of \eqref{RH} is greater than $1$.
        \item $n=4$, $e_1=3$, $e_2=3$ and $e_3\geq 7$. Again, the right-hand side of \eqref{RH} is greater than $1$.
        \item $n=4$, $e_1=e_2=3$ and $e_3=5$. By Proposition \ref{mcm}, $15=\lcm(e_1,e_2,e_3)=\lcm(e_1,e_2,e_3,e_4)$, so either $e_4=5$ or $e_4=15$. Moreover, as $g_0=0$, again by Proposition \ref{mcm} we have $N=\lcm(e_1,e_2,e_3,e_4)=15$. By  \eqref{RH} we obtain that either $g(\cX)=8$ or $g(\cX)=9$, but this contradicts  $N\geq 2g(\cX)+1$.
        \item $n=4$, $e_1=e_2=e_3=3$. By Proposition \ref{mcm} we have $3=\lcm(e_1,e_2,e_3)=\lcm(e_1,e_2,e_3,e_4)$, from which $e_4=3$; again by Proposition \ref{mcm} $N=\lcm(e_1,e_2,e_3,e_4)$ holds. By replacing $e_1=e_2=e_3=e_4=N=3$ and $g_0=0$ in (\ref{RH}) we have $g(\cX)=2$. Whence $N=2g(\cX)-1$, against $N\geq 2g(\cX)+1$.
    \end{itemize}
    \item {\bf If $N$ is even, then either ${\mathbf{n=3}}$ or ${\mathbf{n=4}}$ and $\sigma$ is of type $(0;2,2,g(\cX)+1,g(\cX)+1)$ with $g(\cX)$ even}. Arguing as above we want to prove that either $n=3$ or $n=4$, $e_1=e_2=2$, $e_3=e_4=g(\cX)+1$ and $g(\cX)$ is even.\\
    With the same arguments of the previous item, 
   
        the unique possibility for $n>3$ is $n=4$, $e_1=e_2=2$, and $e_3\geq 3$. Since from Proposition \ref{mcm} we have $\lcm(e_1,e_2,e_3,e_4)=\lcm(e_1,e_2,e_3)$, then either $e_4=e_3$ or $e_4=2e_3$ holds.
    In the latter case, by Proposition \ref{mcm} $N=\lcm(2,e_3,2e_3)=2e_3$, and from \eqref{RH} we have
    $$\dfrac{2g(\cX)-2}{2e_3}=2-\dfrac{1}{2}-\dfrac{1}{2}-\dfrac{1}{e_3}-\dfrac{1}{2e_3};$$
     hence 
    $$2g(\cX)=2e_3-1=N-1,$$
    a contradiction as $N$ is even.
    
    If, instead, $e_4=e_3$, then \eqref{RH} reads 
    \begin{equation}\label{RH1}
    \dfrac{2g(\cX)-2}{N}=2-\dfrac{1}{2}-\dfrac{1}{2}-\dfrac{1}{e_3}-\dfrac{1}{e_3}.
    \end{equation}
    For $N$ there are two possibilities: if $e_3$ is even, then $N=\lcm(2,e_3)=e_3$, whereas if $e_3$ is odd, then $N=\lcm(2,e_3)=2e_3$.\\
    Assume first that $e_3$ is even; then \eqref{RH1} reads
    $$\dfrac{2g(\cX)-2}{e_3}=2-\dfrac{1}{2}-\dfrac{1}{2}-\dfrac{1}{e_3}-\dfrac{1}{e_3},$$
    which implies $g(\cX)=\dfrac{e_3}{2}=\dfrac{N}{2}$ against $N\geq 2g(\cX)+1$.\\
    Finally, if $e_3$ is odd, then from \eqref{RH1} we have
    $$\dfrac{2g(\cX)-2}{2e_3}=2-\dfrac{1}{2}-\dfrac{1}{2}-\dfrac{1}{e_3}-\dfrac{1}{e_3},$$
    which yields $e_3=g(\cX)+1$ and hence $N= 2g(\cX)+2$. In particular, $g(\cX)$ is even in this case.\\
\end{enumerate}
We will continue the proof by showing that if $n=3$ then $\mathcal{X}$ is birationally equivalent to $F(r,s)$ for some primitive pair $(r,s)\in A_N$, whereas if $\sigma$ is of type $(0;2,2,g(\cX)+1,g(\cX)+1)$ and $g(\cX)$ is even then $\mathcal{X}$ is isomorphic to $H_{\lambda}$ for some $\lambda\in\mathbb{K}\setminus\left\{0,1\right\}$ with $N=2g(\cX)+2$.

Recall that $\mathcal{X}\longrightarrow \mathcal{X}/G$ is a cyclic covering of degree $N$ with $n$ ramification points and that $\mathcal X/G$ is the projective line $\mathbb P^1(\mathbb K)$. Let $x$ denote a generator of the function field of $\mathcal X/G$. For an element $\beta$ in $\mathbb K$ let $S_\beta\in \mathbb P^1(\mathbb K)$ denote the point such that $x(S_\beta)=\beta$; also, let $S_\infty$ denote  the only pole of $x$.

Assume first that $n=3$ holds. 
Without loss of generality assume that the three ramification points 
of $\mathcal{X}\longrightarrow \mathcal{X}/G$
are $S_0$, $S_1$ $S_\infty$.
Since $\mathcal{X}$ is a Kummer extension of degree $N$ of $\mathbb{K}(x)$, by (\cite{Stichtenoth}, Prop. 3.7.3 (b)) an equation for $\mathcal X$ is 
$y^N=\alpha(x)$, with 
$$\gcd(ord_{S_0}(\alpha),N)<N,\,\,\,\,\,\gcd(ord_{S_1}(\alpha),N)<N,\,\,\,\,\,\gcd(ord_{S_\infty}(\alpha),N)<N$$
and  $\gcd(ord_P(\alpha),N)=N$ for any other point $P\in \mathbb P^1(\mathbb K)$, $P\notin \{S_0,S_1,S_\infty\}$. 
Now,
\begin{eqnarray}\label{valutazioni}
0=\sum_{\substack{P\in\mathbb{P}^1(\mathbb K) \\
     ord_P(\alpha)\neq 0}}
ord_P(\alpha)=ord_{S_0}(\alpha)+ord_{S_1}(\alpha)+ord_{S_\infty}(\alpha)+\sum_{\substack{P\in\mathbb{P}^1(\mathbb K),P\neq S_0,S_1,S_\infty\\
      ord_P(\alpha)\neq 0\\
      }}ord_P(\alpha).
\end{eqnarray}
Let $R_1,\dots,R_k$ be the points of $\mathbb P^1(\mathbb K)$ where the order of $\alpha$ is different from $0$ and it is a multiple of $N$, and put $\beta_i=x(R_i)$.
Therefore,
$$ord_{R_1}(\alpha)=Nh_1,\dots, ord_{R_k}(\alpha)=Nh_k$$
for suitable integers $h_1,\dots,h_k.$
Then \eqref{valutazioni} becomes
$$0=ord_{S_0}(\alpha)+ord_{S_1}(\alpha)+ord_{S_\infty}(\alpha)+\sum_{i=0}^{k}ord_{R_i}(\alpha).$$
Since $ord_{R_1}(\alpha)=Nh_1$, $\alpha$ will be of the form $\alpha(x)=(x-\beta_1)^{Nh_1}\alpha_1(x)$ with $ord_{R_1}(\alpha_1)=0$.\\
Consider
\begin{equation}\label{sistema}
\begin{cases}
y^{\prime}=\dfrac{y}{(x-\beta_1)^{h_1}}\\
x^{\prime}=x
\end{cases},
\end{equation}
then
$$(y^{\prime})^N=\dfrac{y^N}{(x-\beta_1)^{Nh_1}}=\alpha_1(x^{\prime}).$$
Therefore, \eqref{sistema} defines  an isomorphism between the curves $y^N=\alpha(x)$ and $y^N=\alpha_1(x)$.\\
Moreover, for each $P\in \mathbb{P}^1(\mathbb K)$ we have
$$ord_P(\alpha)=Nh_1\cdot ord_P(x-\beta_1)+ord_P(\alpha_1),$$
whence $ord_P(\alpha)=ord_P(\alpha_1)$ for any $P$ different from $R_1$ and $S_\infty$. So, up to replacing $\alpha(x)$ with $\alpha_1(x)$, we may assume $ord_{R_1}(\alpha)=0$.
By repeating this process for any $i=1,\ldots,k$, we obtain that, up to birational equivalence, a defining equation for $\mathcal{X}$ is $y^N=\alpha(x)$ with $ord_P(\alpha)\neq 0$ just in $\{S_0,S_1,S_\infty\}$. 
Also, either $\alpha$ has two zeros and one pole, or it has two poles and one zero.
Actually, if $\alpha$ has two poles and one zero we can replace $y$ with $y'=\dfrac{1}{y}$ and $\mathcal{X}$ with the curve defined by $(y^\prime)^N=\dfrac{1}{\alpha}$.
So, we can assume that $\alpha$ has $S_0$ and $S_1$ as zeros, and $S_\infty$ as the unique pole, whence $\alpha(x)=f(x)$ with $f(x)\in\mathbb{K}[x]$.
Now, if $r:=ord_{S_0}(f(x))$ and $s:=ord_{S_1}(f(x))$, then $f(x)=x^r(x-1)^s$ can be assumed. This proves our claim for $n=3$, as $\gcd(r,s,N)=1$ follows from the irreducibility of the equation.

Suppose now that $n=4$, $\sigma$ is of type $(0;2,2,g(\cX)+1,g(\cX)+1)$ and $g(\cX)$ even. The Riemann-Hurwitz genus formula reads
$$2g(\cX)-2=-2N+\dfrac{N}{2}+\dfrac{N}{2}+\dfrac{N}{g(\cX)+1}g(\cX)+\dfrac{N}{g(\cX)+1}g(\cX),$$
which implies
$$N=2g(\cX)+2.$$
We can assume that $\pi:\mathcal{X}\longrightarrow \mathcal{X}/G$ ramifies at $S_\theta,S_0,S_1,S_\infty$ with $\theta\in\mathbb{K}\setminus\left\{0,1\right\}$ and that
$$\pi^{-1}(S_\theta)=\left\{P,\sigma(P),\dots,\sigma^{g(\cX)}(P)\right\},\,\,\,\,\,\pi^{-1}(S_1)=\left\{Q,\sigma(Q),\dots,\sigma^{g(\cX)}(Q)\right\},$$
$$\pi^{-1}(S_0)=\left\{P_0,\sigma(P_0)\right\},\,\,\,\,\,\pi^{-1}(S_\infty)=\left\{P_{\infty},\sigma(P_{\infty})\right\}.$$
Now, if we let $\tau:=\sigma^{g(\cX)+1}$, then $\tau$ fixes exactly the set of points 
$$\left\{P,\sigma(P),\dots,\sigma^{g(\cX)}(P),Q,\sigma(Q),\dots,\sigma^{g(\cX)}(Q)\right\}$$
and the Riemann-Hurwitz formula for the covering $\mathcal{X}\longrightarrow \mathcal{X}/\langle\tau\rangle$ reads
$$2g(\cX)-2=2(2g'-2)+(2g(\cX)+2),$$
where $g'$ is the genus of $\mathcal{X}/\langle\tau\rangle$. Hence $g'=0$ holds and $\mathbb{K}(\mathcal{X}/\langle\tau\rangle)=\mathbb{K}(t)$ for some  $t\in \mathbb K(\mathcal X)$ and $\mathcal{X}$ is hyperelliptic. Also,
$$\div(t)=P_0+\sigma(P_0)-P_{\infty}-\sigma(P_{\infty}).$$

By \cite[Proposition 6.2.3 (c)]{Stichtenoth}, $\mathcal{X}$ has equation $y^2=f(t)$, for a suitable $f\in\mathbb{K}[t]$ whose zeros are exactly the $2g(\cX)+2$ places of $\mathbb{K}(t)$ that are ramified in the extension $\mathbb{K}(t,y)\mid\mathbb{K}(t)$. So, for 
$a_i=t(\sigma^i(P))$ and $b_i=t(\sigma^i(Q))$ we have that
 $\left\{a_0,\dots,a_{g(\cX)},b_0,\dots,b_{g(\cX)}\right\}$ are the zeros of $f$ and $\mathcal{X}$ has equation
\begin{equation}\label{eqcurva}
y^2=\prod_{i=0}^{g(\cX)}(t-a_i)(t-b_i).
\end{equation}
Note that since
$$\div(\sigma^{*}(t))=\div(t) \,\,\text{ and }\,\, (\sigma^{g(\cX)+1})^{*}(t)=t$$
we have that $\sigma^{*}(t)=\zeta t$ where $\zeta $ is a primitive $(g(\cX)+1)$-th root of unity.
Also, $\div(\sigma^{*}(t-a_i))=\sigma(\div(t-a_i))=\div(t-a_{(i+1 \pmod {g(\mathcal X)}})$ and $\div(\sigma^{*}(t-b_i))=\sigma(\div(t-b_i))=\div(t-b_{(i+1 \pmod {g(\mathcal X)}})$. Then, $\sigma^{*}(t-a_i)=\xi(t-a_{(i+1 \pmod {g(\mathcal X)}})$ and $\sigma^{*}(t-b_i)=\eta(t-b_{(i+1 \pmod {g(\mathcal X)}})$ for some $\xi, \eta\in\mathbb{K}.$
On the other hand, 
$$\sigma^{*}(t-a_i)=\zeta t-a_i=\zeta \left(t-\dfrac{a_i}{\zeta }\right)
$$
and 
$$
\sigma^{*}(t-b_i)=\zeta t-b_i=\zeta \left(t-\dfrac{b_i}{\zeta }\right).$$
Thus
$$\xi(t-a_{(i+1 \pmod {g(\cX)}})=\zeta \left(t-\dfrac{a_i}{\zeta }\right)
$$
and
$$
\eta(t-b_{(i+1 \pmod {g(\cX)}})=\zeta \left(t-\dfrac{b_i}{\zeta }\right),$$
whence $\xi=\eta=\zeta $,  $a_{(i+1 \pmod {g(\cX)}}=\dfrac{a_i}{\zeta }$ and $b_{(i+1 \pmod {g(\cX)}}=\dfrac{b_i}{\zeta }$.\\
So, we obtain $a_i=\dfrac{a_0}{\zeta ^{i}}$ and $b_i=\dfrac{b_0}{\zeta ^{i}}$, and hence \eqref{eqcurva} reads
$$y^2=\prod_{i=0}^{g(\cX)}\left(t-\dfrac{a_0}{\zeta ^{i}}\right)\left(t-\dfrac{b_0}{\zeta ^{i}}\right).$$
Finally, without loss of generality we can assume $a_0=1$ and $b_0=\lambda$, so that \eqref{eqcurva} becomes 
$$y^2=(t^{g(\cX)+1}-1)(t^{g(\cX)+1}-\lambda),$$
which proves the claim.
\end{proof}

\section*{Acknowledgements}
The authors thank the Italian National Group for Algebraic and Geometric Structures and their Applications (GNSAGA—INdAM)
which supported the research. The third author is funded by the project ``Metodi matematici per la firma digitale ed il cloud computing" (Programma Operativo Nazionale (PON) “Ricerca e Innovazione” 2014-2020, University of Perugia).

\section*{Declarations}
{\bf Conflicts of interest.} The authors have no conflicts of interest to declare that are relevant to the content of this
article.


\begin{thebibliography}{}

\bibitem{ABFG} N. Anbar, D. Bartoli, S. Fanali and M. Giulietti, \emph{On the size of the automorphism group of a plane algebraic curve}, Journal of Pure and Applied Algebra 217 (2013), 1224-1236.

\bibitem{Arakelian} N. Arakelian and P. Speziali \emph{Algebraic curves with automorphism groups of large prime order}, Mathematische Zeitschrift 299 (2021), 2005-2028.

\bibitem{Dickson} L.E. Dickson, \emph{Linear Groups}, with an Exposition of the Galois Field Theory, Teubner, Leipzig, (1901).

\bibitem{Fulton} W. Fulton, \emph{Algebraic curves}, Advanced Book Classics, Addison-Wesley Publishing Company Advanced Book Program, Redwood City, CA, (1989).

\bibitem{Large3} M. Giulietti and G. Korchm\'aros \emph{Large $3$-groups of automorphisms of algebraic curves in characteristic $3$}, arXiv:1312.5108, (2013).

\bibitem{AdvancesGK} M. Giulietti and G. Korchm\'aros \emph{Algebraic curves with many automorphisms}, Advances in Mathematics 349 (2019), 162--211.

\bibitem{henn} H.W. Henn, \emph{Funktionenkorper mit grosser Automorphismengruppe}, J. Reine Angew. Math. 302 (1978), 96-115.

\bibitem{HKT} J.W.P Hirschfeld, G. Korchm\'aros and F. Torres, \emph{Algebraic Curves Over a Finite Field}, Princeton Univ. Press, Princeton and Oxford, (2008).

\bibitem{Homma} M. Homma, \emph{Automorphisms of prime order of curves}, Manuscripta Math. 33 (1980/81), no. 1, 99-109.

\bibitem{Sasaki} S. Irokawa and R. Sasaki, \emph{On a family of quotients of Fermat curves}, Tsukuba Journal of Mathematics 19 (1995), 121-139.

\bibitem{kmoddpower} G. Korchmáros and M. Montanucci, \emph{Large odd prime power order automorphism groups of algebraic
curves in any characteristic}, Journal of Algebra 547 (2020) 312–344.

\bibitem{kmordinarie} G. Korchmáros and M. Montanucci, \emph{Ordinary algebraic curves with many automorphisms in positive
characteristic}, Algebra Number Theory 13 (1) (2019) 1–18.

\bibitem{LiaTimpanella} S. Lia and M. Timpanella, \emph{Bound on the order of the decomposition groups of an algebraic curve in positive characteristic}, Finite Fields and Their Applications 69 (2021), 101771.

\bibitem{msordinarie} M. Montanucci and P. Speziali, \emph{Large automorphism groups of ordinary curves in characteristic 2}, Journal of
Algebra 526 (2019) 30–50.

\bibitem{MZ} M. Montanucci and G. Zini, \emph{Generalized Artin-Mumford curves over finite fields}, Journal of
Algebra  485 (2017), 310-331.

\bibitem{NAKA} S. Nakajima, \emph{$p$-ranks and automorphism groups of algebraic curves}, Trans. Amer. Math. Soc. 303 (1987), 595--607 

\bibitem{MR}
M. Matignon and  M. Rocher, \emph{Smooth curves having a large automorphism $p$-group in characteristic $p>0$}. Algebra {\&} Number Theory 2(8) (2008), 887--926


\bibitem{Serre} J.P. Serre, \emph{Local Fields}, Graduate Texts in Mathematics, vol. 67, Springer, New York, (1979), viii+241
pp.

\bibitem{silverman2009} J.H.~Silverman, \emph{The arithmetic of elliptic curves},
Second edition, Graduate Texts in Mathematics, 106 Springer, Dordrecht, (2009). xx+513 pp.

\bibitem{Singh} B. Singh, \emph{On the group of automorphisms of a function field of genus at least two}, J. Pure Appl. Algebra 4 (1974) 203--229.

\bibitem{Stichtenoth}H.~Stichtenoth, \emph{ Algebraic Function
Fields and Codes},  Springer, Berlin, (1993), x+260 pp.

\end{thebibliography}
\end{document}